\documentclass[12pt]{amsart}
\usepackage{amscd}
\usepackage{amssymb}
\usepackage{a4wide}
\usepackage{amstext}
\usepackage{amsthm}
\usepackage{mathrsfs}
\usepackage{relsize}

\usepackage{tikz}
\usepackage{float}
\DeclareGraphicsExtensions{.pdf,.png,.jpg}
\usetikzlibrary{graphs}
\usetikzlibrary{arrows,positioning,calc}

\usepackage[final]{hyperref}
\hypersetup{unicode= false, colorlinks=true, linkcolor=blue,
anchorcolor=blue, citecolor=green, filecolor=red, menucolor=blue,
pagecolor=blue, urlcolor=blue} 
\numberwithin{equation}{section}

\newcommand{\Z}{\mathbb{Z}}

\renewcommand{\phi}{\varphi}

\newcommand{\vep}{\varepsilon}
\newcommand{\co}{\mathbb{C}}
\newcommand{\N}{\mathbb{N}}

\newtheorem{Thm}{Theorem}[section]
\newtheorem{theorem}[Thm]{Theorem}

\newtheorem{proposition}[Thm]{Proposition}

\newtheorem{remark}[Thm]{Remark}
\newtheorem{example}[Thm]{Example}

\begin{document}
\sloppy

\title[Hypercyclic shifts on lattice graphs]{Hypercyclic shifts on lattice graphs}

\author[A. Baranov, A. Lishanskii, D. Papathanasiou]{Anton Baranov, Andrei Lishanskii, Dimitris Papathanasiou}

\address{Anton Baranov\\
	Department of Mathematics and Mechanics, St. Petersburg State University, 28, Universitetskii prosp., St. Petersburg, 198504, Russia}
\email{anton.d.baranov@gmail.com}

\address{Andrei Lishanskii\\
St. Petersburg Department of Steklov Mathematical Institute of Russian Academy of Sciences; 
Leonard Euler International Institute at St. Petersburg, 
27, Fontanka, St. Petersburg, 191023, Russia }
	\email{lishanskiyaa@gmail.com}
	
	\address{Dimitris Papathanasiou\\
	Sabanci University, Orta Mahalle, Universitesi Cd. No: 27, 34956 Tuzla/Istanbul, Turkey }
	\email{d.papathanasiou@sabanciuniv.edu}

\thanks{The work of A. Baranov in Sections 3 and 4 is supported by the Russian Science Foundation grant
24-11-00087. The work of A. Lishanskii is supported by the Ministry of Science and Higher Education of the
Russian Federation, agreement 075-15-2022-289.}

\maketitle

\begin{abstract}
Recently K.-G.~Grosse-Erdmann and D.~Papathanasiou described hypercyclic shifts in weighted spaces
on directed trees. In this note we discuss several simple examples of graphs which are not trees, e.g., the lattice graphs,
and study hypercyclicity of the corresponding backward shifts.
\end{abstract}

\section{Introduction}

Let $G= (V,E)$ be a connected directed graph consisting of a countable set 
$V$ of vertices where $E \subset V\times V$ is its set of edges. 
For $v,u\in V$ we write $v \to u$ if $(v, u) \in E$. 
Given a vertex $v\in V$ we denote by $Chi(v)$ the set of its ``children'':
$$
Chi(v) =\{u\in V:\ v\to u \}.
$$
More generally, for $n\ge 1$, we put $Chi^n(v) = Chi(Chi(. . . (Chi(v))))$ ($n$ times). 
Analogously, we denote by $Par(v)$ the set of those $u$ for which $v\in Chi(u)$ (``parents'' of $v$), and
$Par^n(v) = Par(Par(. . . (Par(v))))$ $n$ times. 
In what follows we will assume  that each vertex in the graph $G$ has a finite degree, that is, the set
$Chi(v)$ is finite for any $v\in V$.

For any graph $G$ and function $f : V \to \co$ we define the (unweighted) backward shift 
$$
   (B f) (v) = \sum\limits_{u \in Chi (v)} f(u), \qquad v \in V.
$$

Recall that a continuous linear operator $T$ on a separable Banach space 
$X$  is said to be  \textit{hypercyclic} if there exists $x \in X$ such that the set 
$\{T^n x:\, n\in\mathbb{N}_0\}$ is dense in $X$
(here $\mathbb{N}_0 = \{0,1,2, \dots\}$). The operator $T$
is said to be {\it weakly mixing} if $T \oplus T$ is hypercyclic on $X \oplus X$.
Finally,  $T$ is said to be {\it mixing} if for any nonempty open sets $U$ and $V$
there exists $N$ such that $T^n(U)\cap V \ne \emptyset$ for $n\ge N$.
It is well known that mixing implies weak mixing and weak mixing implies hypercyclicity. 
For the theory of hypercyclic operators see \cite{bm, gp}.

Weighted shifts are among the most well-known examples of hypercyclic operators.
Hypercyclic shifts on weighted spaces $\ell^p(\N, \mu)$ and $\ell^p(\Z, \mu)$ of one-sided and two-sided sequences
were described by H.N.~Salas \cite{sal}. Note that $\N$ and $\Z$ can be considered as simplest examples of a rooted
and unrooted tree respectively. 

We will consider the hypercyclicity properties of the backward shift on the standard spaces $\ell^p(V, \mu)$
and $c_0(V, \mu)$ on $V$. Let $\mu = (\mu_v)_{v\in V}$ be a family of non-zero (real or complex, but not necessarily positive) numbers,
called a weight. For $1\le p<\infty$ put 
$$
\ell^p(V, \mu) = \{f:V\to\co:\ \|f\|^p_{\ell^p(V, \mu)} = \sum_{v\in V} |f(v)\mu_v|^p <\infty\}.
$$
The space $c_0(V, \mu)$  is defined as
$$
c_0(V, \mu) = \{f:V\to\co:\ \forall \vep>0 \ \exists F\subset V, \text{finite, such that}\ 
|f(v) \mu_v| <\vep, \ v\in V\setminus F\};
$$
it is equipped with the sup-norm $ \|f\|_{c_0(V, \mu)} = \sup_{v\in V} |f(v)\mu_v|$.

For the case of directed trees (both rooted or not) without leaves, a solution of the hypercyclicity problem was obtained by 
K.-G.~Grosse-Erdmann and the third author \cite{grpap}. Let us formulate their result for a rooted tree in the case
$1<p<\infty$.
\medskip
\\
{\bf Theorem} \cite[Theorem 4.3]{grpap}.
{\it Let $G=(E,V)$ be a rooted directed tree, 
$1<p<\infty$, $1/p+1/q=1$. The operator $B$ on $\ell^p(V, \mu)$ is hypercyclic if and only if 
it is weakly mixing and if and only if 
there is an increasing sequence $(n_k)$  of positive integers such that, for each $v\in V$,}
$$
\sum_{u\in Chi^{n_k}(v)} |\mu_u|^{-q}\to \infty, \qquad k\to \infty.
$$

Previously hypercyclicity of shifts on directed trees was studied by R.A. Mart\'inez-Avenda\~{n}o \cite{mar},
while in a recent preprint \cite{gr1} chaotic weighted shifts on trees are characterized.
For further results about shifts on trees see \cite{aa, lp}.

It should be mentioned that there exist parallel and essentially equivalent theories of weighted shifts on graphs:
one can consider unweighted shift in a space (say, $\ell^p$) with a weight or a weighted shift (with weights ascribed to the edges)
in an unweighted space. In \cite{grpap} both of these viewpoints are considered; here we consider unweighted shifts only. 

The aim of this note is to study  
hypercyclicity of the backward shift for some concrete simple examples of graphs which are not trees. 
Let us mention the following general question: {\it to describe  those directed graphs for which there exist a measure $\mu$ 
on $V$ such that $B$ is hypercyclic on  $\ell^p(V, \mu)$.} One of the trivial obstacles
is (as noted already in \cite{grpap}) the existence of vertices $v\in V$ such that $Chi(v) = \emptyset$. 
For another simple example of a graph which does not carry a hypercyclic weighted shift see Section \ref{ex}.
It seems that existence of cycles also makes it more difficult (but not impossible) to have a hypercyclic weighted shift.
\medskip
\\
\bigskip


\section{Model examples}

One can expect that the next natural class of graphs for which the hypercyclic shifts can be 
described is the class of Cartesian products of trees. Recall that given two graphs $G = (V(G), E(G))$, 
$H =  (V(H), E(H))$
(directed or not), the vertex set of $G\times H$ is the Cartesian product $V(G) \times V(H)$
and two vertices $(u,u')$ and $(v,v')$ are adjacent in $G\times H$ if and only if either
$u = v$ and $(u',v') \in E(H)$, or $u' = v'$ and $(u,v) \in E(G)$.

Even in the class  of Cartesian products of trees hypercyclicity of shifts seems to be a difficult problem.
In this note we consider the simplest cases of directed lattice graphs, i.e.,
Cartesian products $ [1,2, \dots,m]\times \N$, $ [1,2, \dots,m]\times \Z$ and $\N \times \N$, 
where $[1,2, \dots,m]$, $\N$, $\Z$ are considered as directed graphs with natural orientations of the edges.
Namely, for $m\in \N$, consider the graph $G_m = (V_m, E_m)$ such that 
$V_m= (v_{i,j})_{1\le i \le m, j\ge 1} \cong [1,2, \dots,m]\times \N$ and
\begin{equation}
\label{rty}
(v_{i,j}, u) \in E_m \ \Longleftrightarrow \ u = v_{i,j+1} \quad \text{or} \quad u= v_{i+1,j}, \ 1\le i\le m-1.
\end{equation}
Below we show the picture of the graph $G_m$ with $m=3$:
\medskip

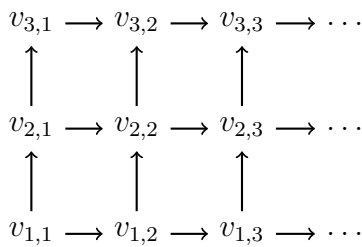
\begin{center}
\begin{tikzpicture}[node distance={14mm}, thick, main/.style = {draw, circle}] 
\node (1) {$v_{1,1}$}; 
\node (2) [above of=1] {$v_{2,1}$}; 
\node (3) [above of=2] {$v_{3,1}$}; 
\node (4) [right of=1] {$v_{1,2}$}; 
\node (5) [above of=4] {$v_{2,2}$}; 
\node (6) [above of=5] {$v_{3,2}$}; 
\node (7) [right of=4] {$v_{1,3}$}; 
\node (8) [above of=7] {$v_{2,3}$}; 
\node (9) [above of=8] {$v_{3,3}$}; 
\node (10) [right of=7] {$\dots$}; 
\node (11) [above of=10] {$\dots$}; 
\node (12) [above of=11] {$\dots$}; 
\draw[->] (1) -- (2); 
\draw[->] (2) -- (3); 
\draw[->] (1) -- (4); 
\draw[->] (4) -- (5); 
\draw[->] (5) -- (6); 
\draw[->] (2) -- (5); 
\draw[->] (3) -- (6); 
\draw[->] (4) -- (7); 
\draw[->] (7) -- (10); 
\draw[->] (5) -- (8); 
\draw[->] (8) -- (11); 
\draw[->] (6) -- (9); 
\draw[->] (9) -- (12); 
\draw[->] (7) -- (8); 
\draw[->] (8) -- (9); 
\end{tikzpicture} 
\end{center} 

Analogously, we define the graph $\tilde G_m = (\tilde E_m, \tilde V_m)$ with 
$\tilde V_m= (v_{i,j})_{1\le i \le m, j\in \Z} \cong [1,2, \dots,m]\times \Z$, whose vertices are also given by 
\eqref{rty} but with $j\in\Z$.

Similarly, we consider the lattice graphs $G_\infty$ and $\Tilde G_\infty$  such that $V_\infty \cong \N \times \N$,
$\Tilde V_\infty \cong \Z \times \N$  and all edges are of the form $(v_{i,j}, v_{i,j+1})$ or $(v_{i,j}, v_{i+1,j})$. 
Below we give the picture of the graph $G_\infty$:
 
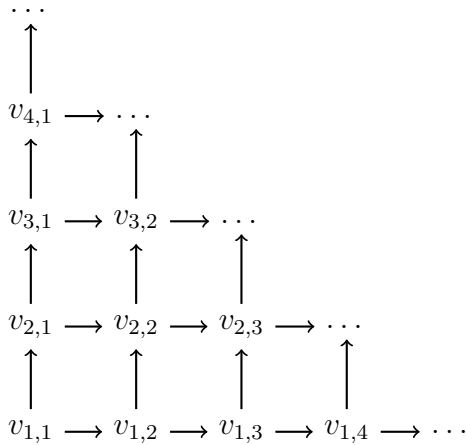
 \begin{center}
\begin{tikzpicture}[node distance={14mm}, thick, main/.style = {draw, circle}] 
\node (1) {$v_{1,1}$}; 
\node (2) [above of=1] {$v_{2,1}$}; 
\node (3) [above of=2] {$v_{3,1}$}; 
\node (4) [right of=1] {$v_{1,2}$}; 
\node (5) [above of=4] {$v_{2,2}$}; 
\node (6) [above of=5] {$v_{3,2}$}; 
\node (7) [right of=4] {$v_{1,3}$}; 
\node (8) [above of=7] {$v_{2,3}$}; 
\node (9) [above of=8] {$\dots$}; 
\node (10) [right of=7] {$v_{1,4}$}; 
\node (11) [above of=10] {$\dots$}; 
\node (13) [above of=3] {$v_{4,1}$};
\node (14) [above of=13] {$\dots$};
\node (15) [right of=13] {$\dots$}; 
\node (16) [right of=10] {$\dots$}; 
\draw[->] (1) -- (2); 
\draw[->] (2) -- (3); 
\draw[->] (1) -- (4); 
\draw[->] (4) -- (5); 
\draw[->] (5) -- (6); 
\draw[->] (2) -- (5); 
\draw[->] (3) -- (6); 
\draw[->] (4) -- (7); 
\draw[->] (7) -- (10); 
\draw[->] (5) -- (8); 
\draw[->] (8) -- (11); 
\draw[->] (6) -- (9); 
\draw[->] (7) -- (8); 
\draw[->] (8) -- (9); 
\draw[->] (3) -- (13); 
\draw[->] (13) -- (14); 
\draw[->] (10) -- (16); 
\draw[->] (10) -- (11); 
\draw[->] (13) -- (15); 
\draw[->] (6) -- (15); 
\end{tikzpicture} 
\end{center} 

It is obvious that $B$ is bounded on $\ell^p(V,\mu)$, $1 \le p <\infty$, or on $c_0(V,\mu)$ for each of the above lattice graphs
if and only if there exists $C>0$ such that for all admissible $(i,j)$
\begin{equation}
\label{bdd}
|\mu_{v_{i,j}}| \le C \min (|\mu_{v_{i+1,j}}|, |\mu_{v_{i,j+1}}|).
\end{equation}

We start with a hypercyclicity/weak mixing criterion for the graph $G_m$.

\begin{theorem} 
\label{t1}
Let $B$ be the backward shift on $G_m$, $m\in \N$, and let $X$ be any of the spaces 
$\ell^p(V_m,\mu)$, $1 \le p <\infty$, or $c_0(V_m,\mu)$. Assume that $B$ is bounded on $X$. Then 
the following are equivalent\textup:
\smallskip

\textup(i\textup) $B$ is hypercyclic on  $X$\textup;
\smallskip

\textup(ii\textup) $B$ is weakly mixing  on  $X$\textup;
\smallskip

\textup(iii\textup) there exists an increasing sequence $(n_k)$  of positive integers such that 
for any $1\le i\le m$, $j\ge 1$,
\begin{equation}
\label{gop}
n_k^{m-i} |\mu_{v_{i, j + n_k}}| \to 0, \qquad k \to \infty.
\end{equation}

\textup(iv\textup) there exists an increasing sequence $(n_k)$  of positive integers such that 
for any $1\le i\le m$,
\begin{equation}
\label{gop5}
n_k^{m-i} |\mu_{v_{i, n_k}}| \to 0, \qquad k \to \infty.
\end{equation}
\end{theorem}

We see that there is a substantial difference with the case of a tree, where it is sufficient that $\mu$ 
tends to 0 along a subset of children of any vertex. 
A novel feature is that the weights should tend to zero with various speed depending on the layer.

A similar result is true for the unrooted case
$\tilde G_m = [1,2, \dots,m]\times \Z$.
As in the classical case of $\Z$ the weights should tend to zero along some symmetric subsequences.

\begin{theorem} 
\label{un1}
Let $B$ be the backward shift on $\tilde G_m$, $m\in \N$, and let $X$ be any of the spaces 
$\ell^p(\tilde V_m,\mu)$, $1 \le p <\infty$, or $c_0(\tilde V_m,\mu)$. Assume that $B$ is bounded on $X$. Then 
the following are equivalent\textup:
\smallskip

\textup(i\textup) $B$ is hypercyclic on  $X$\textup;
\smallskip

\textup(ii\textup) $B$ is weakly mixing  on  $X$\textup;
\smallskip

\textup(iii\textup) there exists an increasing sequence $(n_k)$  of positive integers such that 
for any $1\le i\le m$, $j\in \Z$,
\begin{equation}
\label{ung}
n_k^{m-i} (|\mu_{v_{i, j + n_k}}|  + |\mu_{v_{i, j - n_k}}|) \to 0, \qquad k \to \infty.
\end{equation}
\end{theorem}

We see that when the number of rows in the graph $G_m$ or $\tilde G_m$ grows the conditions on the weight become more
and more restrictive. It is therefore a bit suprising that in the case of doubly infinite lattice $\N \times \N$ 
the backward shift is weakly mixing under much milder conditions on the weight. 
The following result gives a sufficient condition of hypercyclicity/weak mixing which is sharp in a rough exponential scale. 
It is an interesting open problem to find a necessary and sufficient condition for hypercyclicity
of the shift on $\N \times \N$.

\begin{theorem}
\label{nxn}
Let $B$ be the backward shift on $G_\infty =\N \times \N$ and let $X$ be any of the spaces 
$\ell^p(V_\infty,\mu)$, $1 \le p <\infty$, or $c_0(V_\infty,\mu)$. Assume that $B$ is bounded on $X$.
\medskip

1. If 
$$
\limsup_{i+j \to\infty} |\mu_{v_{i, j}}|^{1/(i+j)} <2,
$$
then $B$ is mixing \textup(in particular, weakly mixing and hypercyclic\textup) on  $X$.
\medskip

2. Assume that there is $c>0$ such that
\begin{equation}
\label{who}
|\mu_{v_{i, j}}| \ge c2^{i+j}.
\end{equation}
Then $B$ on $G_\infty =\N \times \N$ is not hypercyclic.
\end{theorem}

We leave the case of lattices $\N \times \Z$ and $\Z \times \Z$ for a future research. 
It is clear that in this case the class of weights for which $B$ is  weak mixing or hypercyclic 
must be large as in the case of $G_\infty$.

Our last result applies to the case when the weght $\mu$ on the lattice depends only on one of the coordinates.w
In this case, making use of recent results of Q.~Menet and the third author \cite{mp} we show, by reducing
the problem to a theorem of Salas, that the backward shift is always mixing whenever it is bounded. We would like to thank Q. Menet for sharing his insights on identifying the backward shift on $G_{\infty}$ with a generalized shift in the sense of \cite{mp}.

\begin{theorem}
\label{onecoord}
Assume that the weight $\mu$ on the graph $G_\infty$ or $\Tilde G_\infty$
depends on one coordinate only, that is, $\mu_{v_{i,j}}=\mu_i$ for any $j\in \mathbb{N}$ or,
respectively, $j\in \mathbb{Z}$. 
Let $X$ be any of the spaces 
$\ell^p(V,\mu)$, $1 \le p <\infty$, or $c_0(V,\mu)$, where $V = V_\infty$
or $V=\Tilde V_\infty$. If the backward shift $B$ is bounded on $X$, then it is mixing.  
\end{theorem}
\bigskip


\section{Proof for the rooted case}

As in \cite{grpap} we use the classical Hypercyclicity Criterion (see, e.g., \cite[Theorem 3.12]{gp}). 
\medskip
\\
{\bf Theorem (Hypercyclicity Criterion).}
{\it  Let $X$ be a separable Banach space and let $T$ be a bounded operator on X.
Assume that there exist dense subsets $X_0, Y_0$ of $X$, an increasing sequence
$\{n_k\}$ of positive integers, and maps $R_{n_k}: Y_0 \to X$ such that, for any
$x\in X_0$ and $y\in Y_0$,

\textup(i\textup) $T^{n_k} x \to 0$,

\textup(ii\textup) $R_{n_k} y \to 0$,

\textup(iii\textup) $T^{n_k} R_{n_k} y \to y$,
\\
as  $k\to\infty$. Then $T$ is weakly mixing and, in particular, hypercyclic. If furthermore, $T$ satisfies the Hypercyclicity Criterion for the full sequence $\{n\}$, then $T$ is mixing.}
\medskip

We denote by $e_v$ the function such that $e_v(v) =1$ and $e_v(u) =0$, $u\ne v$. 
Note that $B e_v = \sum_{u\in Par(v)} e_u$.

\begin{proof}[Proof of Theorem \ref{t1}, sufficiency of \eqref{gop}] 
First we consider the case of the graph $G_m$.
Note that for $n \geq m$ we have $Chi^{n} (v_{ij}) = \{v_{i, j+n}, v_{i+1, j+n-1}, \ldots, v_{m, j+n+i-m}\}$,
while $Par^n(v_{i, j+n}) = \{v_{i,j}, v_{i-1,j+1}, \dots, v_{1, i+j-1}\}$. Also,
\begin{equation}
\label{sib}
\begin{aligned}
B (e_{v_{i,j}}) & = e_{v_{i-1,j}} + e_{v_{i,j-1}}, \\
B^2 (e_{v_{i,j}}) & = e_{v_{i-2,j}} + 2 e_{v_{i-1,j-1}} + e_{v_{i,j-2}} \\
& \dots \\
B^n (e_{v_{i,j}}) & = \sum_{l=0}^n C_n^l e_{v_{i-l,j-(n-l)}},
\end{aligned}
\end{equation}
for those values of $l$ for which $e_{v_{i-l,j-(n-l)}}$ makes sense. We agree to understand 
$e_{v_{i-l,j-(n-l)}}$ as a zero function if $i \le l$ or $j\le n-l$.  Thus, 
$$
\begin{aligned}
B^n (e_{v_{1,j+n}}) & =  e_{v_{1,j}} \\
B^n (e_{v_{2,j+n}}) & =  n e_{v_{1,j+1}}+ e_{v_{2,j}} \\
& \dots \\
B^n (e_{v_{i,j+n}}) & = \sum_{l=0}^{i-1} C_n^l e_{v_{i-l,j+l}}.
\end{aligned}
$$

Note that the set of finite functions is dense in $X$ 
and that $B^n f =0 $ for any finite $f$ and sufficiently large $n$. 
We need to define the operators $R_{n_k}$ so that 
$R_{n_k} f \to 0$ and $B^{n_k}R_{n_k} f \to f$ on finite sequences. Put
$$
R_{n_k} (e_{v_{1, j}}) := e_{v_{1, j+n_k}}, \qquad 
R_{n_k} (e_{v_{2, j}}) := e_{v_{2, j+n_k}} - n_k e_{v_{1, j+1+n_k}}, 
$$
and, assuming that $R_{n_k}$ is already defined on the layers with numbers $1,\dots, i-1$,
we put
\begin{equation}
\label{sonn}
R_{n_k} (e_{v_{i,j}}) = e_{v_{i, j+n_k}} - \sum_{l=1}^{i-1} C_{n_k}^l R_{n_k}(e_{v_{i-l, j+l}}).
\end{equation}
Then, for any $v$, we have $B^{n_k}R_{n_k}(e_v) = e_v$. 
Note that 
\begin{equation}
\label{sonn1}
R_{n_k} (e_{v_{i,j}})  = \sum_{s=1}^i \alpha_{i,s} e_{v_{s, i+j-s+n_k}}
\end{equation}
for some coefficients $\alpha_{i,s}$ independent on $j$ (but, of course, depending on $n_k$). 
It is easy to show by induction that
\begin{equation}
\label{sonn2}
|\alpha_{i,s}| \le C n_k^{i-s},
\end{equation}
where the constant $C$ may depends on $m$ only. In particular, $\alpha_{i,s} =O(n_k^{m-s})$
for all $1\le i \le m$. 
By the hypothesis \eqref{gop} $n_k^{m-s} \mu_{v_{s, i+j-s+n_k}} \to 0$ 
for any $i,j,s$ and we conclude that $R_{n_k} f \to 0$ for any finite $f$. Thus, $B$ is  weakly mixing.
\end{proof}
\medskip

\begin{proof}[Proof of Theorem \ref{t1}, necessity of \eqref{gop}]
We follow here the idea of the proof from \cite{grpap}. 
If $B$ is hypercyclic, then for any 
$\vep, \delta> 0$ and $K\in \N$ we can choose an arbitrarily large $n$ and $f\in X$ such that $\|f\|_X < \vep$ and 
$$
\big\|B^n f - \sum\limits_{j=1}^{K} e_{v_{m,j}}\big\|_X < \delta \min\limits_{v \in F} |\mu_v|,
$$
where $F = \{(v_{i, j}), 1\le  i \leq m, 1\le j \leq K\}$
(take as $f$ an appropriate multiple of the hypercyclic vector for $B$).  
Then we have, for $1\le j \leq K$,
$$
\big| (B^n f) (v_{m,j}) -1 \big| < \delta, \qquad  \big| (B^n f) (v_{i,j}) \big| < \delta, \qquad 1\le i\le m-1.
$$

Note that for $n>m$
$$
(B^nf)  (v_{i, j}) = \sum_{l=0}^{m-i} C_n^l f(v_{i+l, j+n-l}).
$$
In particular, 
$ (B^n f) (v_{m, j}) = f(v_{m, j+n})$, whence 
$|f(v_{m, j+n}) -1| < \delta$, $1\le j\le K$. Next
$$
\big| (B^n f) (v_{m-1, j})\big| = \big|f (v_{m-1, j+n}) + n f (v_{m, j-1+n})  \big| <\delta,
$$
whence
$$
 |f (v_{m-1, j+n}) +n| < \delta(1+n), \qquad 2\le j\le K.
$$
Continuing the estimates we obtain (e.g., by induction) that for any $l<m$
$$
 \Big|f (v_{m-l, j+n}) + \frac{(-1)^{l+1}}{l!} n^l\Big| <C\delta n^l,  \qquad l+1\le j\le K,
$$
where the constant $C>0$ depends on $m$ only. Thus, if we choose $\delta>0$ to be sufficiently small, we get
$$
|f (v_{i, j+n})| \ge C n^{m-i}, \qquad 1\le i\le m, \ m\le j\le K,
$$
where $C>0$ is (another) constant depending only on $m$. Replacing $n$ by $n+m-1$,
we get $|f (v_{i, j+n})| \ge C n^{m-i}$ for $1\le i\le m$, $1\le j\le K-m+1$,
Recall that $\|f\|_X <\vep$. Since $|\mu_v f(v)| \le\|f\|_X  $, we conclude that 
$n^{m-i} |\mu_{v_{i, j+n}}| \le C\vep$, $1\le i\le m$, $1\le j\le K-m+1$.  

Finally, since $\vep>0 $ and $K\in\N$ were arbitrary, we repeat this procedure for
sequences $\vep_k\to 0$ and $K_k\to \infty$ and find a sequence $n_k\to \infty$ satisfying \eqref{gop}.
\medskip

It remains to show that a formally weaker condition \eqref{gop5} implies \eqref{gop}. From the boundedness
of $B$ (see \eqref{bdd}) it follows that for each $l\ge 0$ we have
$(n_k-l)^{m-i}  |\mu_{v_{i, n_k-l}}| \to 0$. It is now not difficult to show (see \cite[Lemma 4.2]{gp}) that 
there exists an increasing sequence $m_k$ such that $(m_k+j)^{m-i}  |\mu_{v_{i, m_k+j}}| \to 0$ for any $j\ge 0$.
\end{proof}
\bigskip


\section{Proof for the unrooted case}

\begin{proof}[Proof of  Theorem \ref{un1}] 
The proof for the unrooted case is similar to the rooted case. To show the sufficiency of \eqref{ung} we define operators
$R_{n_k}$ by the same formula \eqref{sonn}. Then 
$R_{n_k} f \to 0$ and $B^{n_k}R_{n_k} f \to f$ on finite sequences. In the proof of Theorem \ref{t1} (formulas \eqref{sib})
we have seen that
$$
B^{n_k} (e_{v_{i,j}})  = \sum_{l=0}^{i-1} C_{n_k}^l e_{v_{i-l,j+l -n_k}},
$$
and so \eqref{ung} implies also that $B^{n_k} f \to 0$ on finite sequences. 
It remains to apply Hypercyclicity Criterion.

To prove necessity of \eqref{ung}
consider a finite set $F = \{(v_{i, j}), 1\le  i \leq m, -K \le j \leq K\}$. If
$B$ is hypercyclic, then it is  topologically transitive and so 
for any $\vep \in (0,1/2)$ there exist an arbitrarily large $n$ and $f\in X$ such that
$$
\big\| f - \sum\limits_{j=-K}^{K} e_{v_{m,j}}\big\|_X < \vep \min(1,  \min\limits_{v \in F} |\mu_v|),
$$
and 
\begin{equation}
\label{tve}
\big\|B^n f - \sum\limits_{j=-K}^{K} e_{v_{m,j}}\big\|_X < \vep \min\limits_{v \in F} |\mu_v|.
\end{equation}
We can also assume that $v_{i, j+n}, v_{i, j-n} \notin F$  for any $v_{i, j}  \in F$.
Repeating the arguments from the proof of Theorem \ref{t1} we show that
if $\vep$ is sufficiently small, then $|f (v_{i, j+n})| \ge C n^{m-i}$ for $1\le i\le m$, $-K \le j\le K$, where 
$C>0$ is some constant depending on $m$ only. 
Since $ \big\| f - \sum_{j=-K}^{K} e_{v_{m,j}}\big\|_X < \vep$ and 
$v_{i, j+n} \notin F$  for any $v_{i, j}  \in F$ we conclude that
$|\mu_{v_{i, j+n}} f (v_{i, j+n})| <\vep$ for $v_{i, j}  \in F$. Thus, 
$$
n^{m-i} |\mu_{v_{i, j+n}}| \le C\vep, \qquad 1\le i\le m, \ -K\le j\le K,
$$
for some other constant $C>0$ depending on $m$ only. 

Since,  for $-K \le j\le K$, 
$$
 |f(v_{m,j}) -1| \cdot |\mu_{v_{m,j}}| \le 
\big \|f - \sum\limits_{l=-K}^{K} e_{v_{m,l}}\big\|_X < \vep \min(1,  \min\limits_{v \in F} |\mu_v|),
$$
we conclude that for $-K \le j\le K$ we have
$|f(v_{m,j})| \ge 1/2$ whereas $|f(v_{i,j})|<\vep$, $1\le i\le m-1$.
Now we write $f= \sum_{i=1}^m\sum_{j\in \Z} f(v_{i,j}) e_{v_{i,j}}$.
Then
$$
B^n f =   \sum_{i=1}^m\sum_{j\in \Z} f(v_{i,j}) \sum_{l=0}^{i-1} C_{n}^l e_{v_{i-l,j+l -n}} = 
\sum_{k=1}^m \sum_{j\in \Z}  \Big[ \sum_{i=k}^{m} C_n^{i-k} f(v_{i, j-i+k})\Big] e_{v_{k,j-n}}.
$$
For sufficiently large $n$ one has 
$|\sum_{i=k}^{m} C_n^{i-k} f(v_{i, j-i+k})| \ge C n^{m-k}$, $1\le k\le m$, $-K+m-1 \le j\le K$.
Since  $ v_{k,j-n} \notin F$, $-K\le j\le K$, we conclude  as above that 
$n^{m-k} |\mu_{v_{k, j-n}}| \le C\vep$, $1\le k\le m$, $-K+m-1 \le j\le K$. 

Finally, since $\vep>0 $ and $K\in\N$ were arbitrary, we can find a sequence $n_k\to \infty$
satisfying \eqref{ung}.
\end{proof}
\bigskip


\section{Proof of Theorem \ref{nxn}}

We give two slightly different proofs of Statement 1 in Theorem \ref{nxn}. In the first of them we use 
the  following simple sufficient condition of mixing known as  the Godefroy--Shapiro 
Criterion (see \cite{gosh}, \cite[Corollary 1.10]{bm} or \cite[Theorem 3.1]{gp}), while the second proof 
is based on the Hypercyclicity Criterion. 
\medskip
\\
{\bf Theorem (Godefroy--Shapiro Criterion).} 
{\it If, for a bounded linear operator $T$ on a separable Banach 
space $X$, both $\cup_{|\lambda|<1}{\rm Ker}\,(T-\lambda I)$ 
and $\cup_{|\lambda|>1}{\rm Ker}\,(T-\lambda I)$  span a dense subspace in $X$, then $T$ is mixing. }
\medskip
\\
\begin{proof}[Proof of  Theorem \ref{nxn}]
Throughout the proof we will consider the shift on the lattice $\N_0 \times \N_0$ in place of $\N \times \N$. 
This will slightly simplify the formulas.
\medskip
\\
{\it First proof of Statement 1.}
 For $r\ge 1$, $s\in\co$, consider the function $f_{r,s}$ 
defined on $\N_0 \times \N_0$ by
$$
f_{r,s} (v_{i, j}) = r^{i+2j} s^{i+j},
$$
that is,
$$
f_{r,s} = 
\begin{pmatrix}
\dots & \dots & \dots & \dots &  \dots & \dots \\
r^8 s^4 & \dots & \dots & \dots & \dots &  \dots  \\
r^6s^3 &  r^7 s^4 & \dots & \dots & \dots &  \dots  \\
r^4 s^2 &  r^5s^3 &  r^6 s^4 & \dots &  \dots  &  \dots \\
r^2s &  r^3s^2 &  r^4s^3 & r^5 s^4 & \dots & \dots \\
1 & rs & r^2s^2 & r^3s^3 & r^4 s^4 & \dots 
\end{pmatrix}.
$$
Since $\limsup_{i+j \to\infty} |\mu_{v_{i, j}}|^{1/(i+j)} =q <2$ (note that if $B$ is bounded then $q>0$), it is clear that 
$f_{r,s} \in X$ whenever  $qr^2|s| <1$. On the other hand
$$
B f_{r,s} = s(r^2+r)f_{r,s}.
$$
Thus, the set of eigenvalues contains the disc $\big\{|\lambda| < \frac{r^2+r}{qr^2} = \frac{1}{q}\big(1+\frac{1}{r}\big)\big\}$
for any $r\ge 1$.
Since $q<2$, its radius is greater than 1 when $r$ is sifficiently close to 1. More presicely, choose $\delta\in (0,1/8)$ so small that 
$|s|(r^2+r)>1$ for any $r\in (1, 1+\delta)$ and $\frac{1}{qr^2} -\delta <|s| < \frac{1}{qr^2}$. Put
$$
\begin{aligned}
U_1 & =\Big\{(r,s): \ r\in (1, 1+\delta), \, \frac{1}{qr^2} -\delta <|s| < \frac{1}{qr^2}\Big\}, \\
U_2 & =\{(r,s): \ r\in (1, 1+\delta),\,|s| <\delta\}.
\end{aligned}
$$
To apply the Godefroy--Shapiro criterion, we show that 
the families $\{f_{r,s}:\ (r,s) \in U_1\}$ and $\{f_{r,s}:\ (r,s) \in U_2\}$ are complete in $X$. 

Indeed, let $g=\{g(v_{i, j})\} = \{g_{i,j}\} \in X^*$ be a sequence in the annulator of $\{f_{r,s}:\ (r,s) \in U_1\}$.
Then we have, for any $r\in (1, 1+\delta)$ and $\frac{1}{qr^2} -\delta <|s| < \frac{1}{qr^2}$,
$$
\sum_{n=0}^\infty \sum_{k=0}^n f(v_{k, n-k})\bar g_{k,n-k} = 
\sum_{n=0}^\infty s^n r^n \sum_{k=0}^n  \bar g_{k,n-k} r^{n-k} =0.
$$
Considering this series as a power series with respect to the variable $s$ we conclude that for any $n\in \mathbb{N}_0$
$$
\sum_{k=0}^n  \bar g_{k,n-k} r^{n-k} =0,   \qquad r\in (1, 1+\delta),
$$
whence, obviously, $g_{k,n-k} =0$ for any $n\ge 0, 0\le k\le n$. Completeness of $\{f_{r,s}:\ (r,s) \in U_2\}$ is analogous.
\medskip
\\
{\it Second proof of Statement 1.}
 For $k\geq 0$ we let
    $$
    X_k=\text{span}\{e_{v_{i+j}}: i+j=k\}
    $$
    be the $(k+1)$-dimensional space of sequences with non-zero entries on the diagonal $\{(i,j): i+j=k\}$. The vector subspace
    $$
    Y=\bigcup_{k=0}^{\infty}X_k
    $$
    is dense in the space. Let $\{a_i\}_{i=0}^{\infty}$ be a sequence of pairwise distinct, non-zero real scalars. For $k\geq 0$ we set
    $$
    f_i^k=\sum_{j=0}^ka_i^je_{k-j,j}, \quad 0\leq i\leq k.
    $$
    Notice that the set $\{f_i^k\}_{i=0}^k$ forms a base for $X_k$ since the determinant of the coefficients is a non-zero Vandermonde determinant. For $n\geq 0$, define $R_n$ on $\{f_i^k\}_{i=0}^k$ by
    $$
    R_nf_i^k=\frac{1}{(1+a_i)^n}\sum_{j=0}^{k+n}a_i^je_{k+n-j,j}.
    $$
    Note that 
    $$
    B^n e_{v_{i,j}} = \sum_{l=\max(0, n-j)}^{\min(n, i)} C_n^l e_{v_{i-l, j-(n-l)}}.
    $$
    Therefore, by a direct computation, 
    $$
    B^n R_n f_i^k= \frac{1}{(1+a_i)^n} \sum_{l=0}^n C_n^l \sum_{j=n-l}^{n+k-l} a_i^j e_{v_{k+n-j-l, j-(n-l)}} = 
   \frac{1}{(1+a_i)^n} \sum_{l=0}^n C_n^l \sum_{s=0}^{k} a_i^{n-l+s} e_{v_{k-s,s}}  = f_i^k.
    $$
    Also, $B^nf_i^k=0$, for each $n>k$. Extend $R_n$ linearly to get first 
an operator from $X_k$ to $X_{k+n}$ and then to get an operator on $Y$. It remains to show that $R_nf_i^k \rightarrow 0$. 
    
    Let $\limsup_{i+j\to \infty} |\mu_{v_{i,j}}|^{1/(i+j)}=q<2$ and  choose $a_i\in (q-1,1)$, $i\geq 0$.
    If $X = \ell^p(V_\infty,\mu)$ we get that
    $$
    \|R_n f_i^k\|=\frac{1}{(1+a_i)^n}\left(\sum_{j=0}^{k+1}a_i^{jp}|\mu_{k+n-j,j}|^p \right)^{1/p}
     \leq \frac{Cq^{k+n}}{(1+a_i)^n(1-a_i^p)^{1/p}}\rightarrow 0.
    $$
The Hypercyclicity Criterion now applies for the full sequence $\{n\}$, and ensures that $B$ is mixing. The case $X = c_0(V_\infty,\mu)$ is
analogous. 
\medskip
\\
{\it Proof of Statement 2.} Assume that $B$ is hypercyclic. Then for any $\vep\in(0,1/2)$  
and $N\in \N$ there exists $f\in X$ and $n>N$ such that $\|f\|_X < \vep$ and
$$
\|B^n f - e_{v_{0,0}}\| <\vep |\mu_{v_{0,0}}|.
$$
Since $(B^nf)  (v_{0,0}) = \sum_{k=0}^{n} C_n^k f(v_{k, n-k})$, we have
$$
\bigg|\sum_{k=0}^{n} C_n^k f(v_{k, n-k})\bigg|>1/2.
$$
Let $k_0$ be such that $|f(v_{k_0, n-k_0})| = \max_{0\le k\le n} |f(v_{k, n-k})|$. 
Using the fact that $\sum_{k=0}^{n} C_n^k =2^n$, we conclude that $|f(v_{k_0, n-k_0})|\ge 2^{-n-1}$. 
It follows from \eqref{who} that $\|f\|_X \ge c/2$, a contradiction with $\|f\|_X < \vep$ if $\vep$ is sufficiently small.
\end{proof}

We conclude this section with a necessary condition for hypercyclicity of the backward shift on graphs, which also proves
Statement 2 of Theorem \ref{nxn}. For simplicity we formulate it for $1<p<\infty$. The cases of 
$\ell^1(V_\infty,\mu)$ and $c_0(V_\infty,\mu)$ require an obvious modification. 

\begin{proposition}
\label{nec}
Let $B$ be bounded on $X= \ell^p(V_\infty,\mu)$, $1 < p <\infty$, and let $1/p+1/p'=1$. If $B$ is hypercyclic on $X$,
then there exists an increasing sequence $(n_k)$  of positive integers such that 
for any $i,j\in \N$,
$$
\sum_{l=0}^{n_k} \frac{(C_{n_k}^l)^{p'}}{ |\mu_{v_{i+l, j+n_k-l}}|^{p'} } \to\infty, \qquad k \to \infty.
$$
\end{proposition}

\begin{proof}
The proof is analogous to \cite[Lemma 4.2]{grpap}. If $B$ is hypercyclic, then for any $\vep \in(0, 1/2)$ and  $K, N\in\N$ 
there exist $f\in X$ and $n>N$ such that 
$$
\|f\|_X<\vep, \qquad \big\|B^n f - \sum\limits_{1\le i, j\le K} e_{v_{i,j}}\big\|_X < \vep \min_{1\leq i,j\leq K}|\mu_{v_{i,j}}|.
$$
Since $(B^nf)  (v_{i,j}) = \sum_{l=0}^{n} C_n^l f(v_{i+l, j+n-l})$, we have for $i,j\le K$ 
$$
1/2 \le \bigg| \sum_{l=0}^n C_n^l f(v_{i +l ,j+n-l}) \bigg| \le \|f\|_X 
\bigg( \sum_{l=0}^{n} \frac{(C_{n}^l)^{p'}}{ |\mu_{v_{i+l, j+n-l}}|^{p'} } \bigg)^{1/{p'}} \le \vep
\bigg( \sum_{l=0}^{n} \frac{(C_{n}^l)^{p'}}{ |\mu_{v_{i+l, j+n-l}}|^{p'} } \bigg)^{1/{p'}}.
$$
Repeating the procedure for $K_k\to \infty$ and $\vep_k\to 0$ we find a sequence $n_k$ as required.
\end{proof}

In the case of trees considered in \cite{grpap} this natural condition turns out be also sufficient. 
We do not know whether it is the case for the lattice $\N\times \N$.
\bigskip


\section{Proof of Theorem \ref{onecoord}}

We start with the following simple but useful observation. 

\begin{proposition}
\label{subg}
Let  $G=(V,E)$ be a subgraph of a directed graph $\Tilde G=(\Tilde V, \Tilde E)$ such that
if $v\in V$ and $(v,u) \in \Tilde E$, then $u\in V$ and $(v,u) \in E$, i.e., $G$ includes all edges which start in $V$.
Let $\mu$ be a weight on $\Tilde G$ such that the backward shift $\Tilde B$ is bounded on 
$\ell^p(\Tilde V,\mu)$, $1 \le p <\infty$, or $c_0(\Tilde V,\mu)$. Then the backward shift $B$ on
$\ell^p(V,\mu)$ or $c_0(V,\mu)$ inherits all dynamical properties 
\textup(mixing, weak mixing, hypercyclicity\textup) of $\Tilde B$ on $\ell^p(\Tilde V,\mu)$ or, respectively, $c_0(\Tilde V,\mu)$.
\end{proposition}

\begin{proof}
Consider the restriction operator $R:\ell^p(\Tilde{V},\mu)\rightarrow \ell^p(V,\mu)$, $1\leq p<\infty$, or 
$R:c_0(\Tilde{V},\mu)\rightarrow c_0(V,\mu)$, $R f=f|_{V}$. 
It is clear that $R$ is a surjective operator of norm one, and, by the properties of $G$,
$$
R\circ \Tilde{B}=B\circ R.
$$
This means that $B$ is quasi-conjugate to $\Tilde{B}$ and so  inherits all dynamical properties  of $\Tilde B$.
\end{proof}

In the following proof we will use recent results from \cite{mp} about generalized shifts. 
Given an operator $T$ on the Banach space $X$, the generalized shift $B_T$ is defined on $\prod_{n\in \N} X$ or 
$\prod_{n\in \Z} X$ as
$$
B_T (x_k)_k = (T x_{k+1})_k.
$$
Thus, if $T=I$, then $B_T$ is a usual shift. 

Consider the following spaces:
$$
\ell^p(X, \Z) = \Big\{(x_k)_k \in X^{\Z} : \sum_k \|x_k\|_X^p<\infty\Big\}, 
$$
$$
c_0(X, \Z) = \Big\{(x_k)_k \in X^{\Z} : \lim_{|k|\to \infty} \|x_k\|_X =0\Big\}.
$$
It is shown in \cite[Corollary 2.8]{mp} that $B_T$ on $\ell^p(X, \Z)$, $1\le p <\infty$, or on 
$c_0(X, \Z)$ is hypercyclic if and only if it is weakly mixing and if and only if $T$ is weakly mixing on $X$,
while $B_T$ is mixing if and only if $T $ is mixing \cite[Corollary 3.2]{mp}. 

\begin{proof}[Proof of Theorem \ref{onecoord}]
It follows from Proposition \ref{subg} that it is sufficient to prove the statement for the graph 
$\Tilde{G}_{\infty}=(\Tilde{V}_{\infty},\Tilde{E}_{\infty})$ since
its subgraph $G_\infty$ satisfies the conditions of the proposition. 
Note that the backward shift $B$ on $G_{\infty}$ when considered as an operator on either
$\ell^p(V_{\infty},\mu)$, $1\leq p<\infty$, or $c_0(V_{\infty},\mu)$ is bounded if and only 
if the backward shift $\Tilde{B}$ on $\Tilde{G}_{\infty}$, when considered as an operator on any of 
$\ell^p(\Tilde{V}_{\infty},\mu)$, $1\leq p <\infty$, or $c_0(\Tilde{V}_{\infty},\mu)$ is bounded, which happens 
precisely when  $\sup_{i\in \mathbb{N}}\frac{|\mu_i|}{|\mu_{i+1}|}<\infty$. 

Now set for $k\in \mathbb{Z}$,
$$
D_k=\{v_{i,j}\in \Tilde{V}_{\infty}: i+j=k\}
$$
and notice that, due to the fact that $\mu$ depends only on the first coordinate, 
$\mu_{v_{i,j}} = \mu_i$, we can identify the spaces $\ell^p(D_k,\mu)$, 
$k\in \mathbb{Z}$, with $\ell^p(\mathbb{N},\mu)$, $1\leq p< \infty$, and similarly $c_0(D_k,\mu)$, $k\in \mathbb{Z}$, with $c_0(\mathbb{N},\mu)$ via the identification $v_{i,k-i} \leftrightarrow i$. This allows us to further identify for $1\leq p<\infty$
$$
\ell^p(\Tilde{V}_{\infty})\cong \ell^p(\ell^p(\mathbb{N},\mu),\mathbb{Z})=\{(f_k)_{k \in \Z}\in \ell^p(\mathbb{N},\mu)^{\mathbb{Z}}: \sum_{k=-\infty}^{\infty}\|f_k\|_{\ell^p(\mathbb{N},\mu)}^p<\infty\}
$$
and 
$$
c_0(\Tilde{V}_{\infty})\cong c_0(c_0(\mathbb{N},\mu),\mathbb{Z})=\{(f_k)_{k \in \Z}\in c_0(\mathbb{N},\mu)^{\mathbb{Z}}: \lim_{|k|\rightarrow \infty}\|f_k\|_{c_0(\mathbb{N},\mu)}\rightarrow 0\}
$$
by $(f(v_{i,j}))_{i\in \mathbb{N}, j\in \mathbb{Z}}\mapsto (f_k)_{k\in \Z}  = ((f(v_{i,k-i}))_{i\in \mathbb{N}})_{k\in \mathbb{Z}}$. 

Let $B_0$ be the unweighted unilateral backward shift on either $\ell^p(\mathbb{N},\mu)$, $1\leq p<\infty$, or $c_0(\mathbb{N},\mu)$.
Under the above identifications, we notice that the backward shift $\Tilde{B}$ 
can be viewed as the generalized shift $B_{I+B_0}$ on 
$\ell^p(\ell^p(\mathbb{N},\mu),\mathbb{Z})$ or $c_0(c_0(\mathbb{N},\mu),\mathbb{Z})$, defined by
$$
B_{I+B_0}(f_k)_{k\in \Z}=((I+B_0)(f_{k+1}))_{k \in \Z}.
$$
Indeed, 
$$
(\Tilde Bf)(v_{i, k-i}) =  f(v_{i, k+1-i}) + f(v_{i+1, k-i}) = f(v_{i, k+1-i}) + f(v_{i+1, k+1-(i+1)}).
$$
Therefore, if we write $f_k = (f(v_{i,k-i}))_{i\in \mathbb{N}}$, then
$\Tilde B (f_k)_{k\in Z} = (f_{k+1} +B_0 f_{k+1})_{k\in \Z}$.
By \cite{mp}, the unilateral generalized shift $B_{I+B_0}$ is mixing if and only its symbol $I+B_0$ 
is mixing. Since by a theorem of Salas \cite{sal} (see also \cite[Theorem 8.2]{gp}), 
$I+B_0$ is always mixing on 
$\ell^p(\mathbb{N},\mu)$, $1\leq p<\infty$, or on $c_0(\mathbb{N},\mu)$ provided it is 
bounded, we conclude that $\Tilde{B}$ is mixing. 
\end{proof}

\begin{remark}
{\rm 
If we drop the assumption that $\mu$ depends on one coordinate only and if $\Tilde{\mu}$ is a weight on $\Tilde{V}_{\infty}$ that extends $\mu$ and makes $\Tilde{B}$ bounded, then we can still repeat the above argument and see
$\Tilde{B}$ as the bilateral generalized shift $B_{I+B_0}$ (which is bounded) but in this case the symbol 
$I+B_0$ on the $n$-th coordinate is an operator from $\ell^p(\mathbb{N},\Tilde{\mu}|_{D_{n+1}})$ to 
$\ell^p(\mathbb{N},\Tilde{\mu}|_{D_n})$. In this case Salas' result cannot be used to conclude that $I+B_0$ is mixing.

Also, if we consider the shift on the lattice $\Z\times \Z$ in the case when $\mu$ depends on one coordinate only,
we can identify it with $B_{I+B_0}$, where $B_0$ is the usual bilateral shift on $\ell^p(\mathbb{Z},\mu)$. However, in the bilateral
case it is not known whether $I+B_0$ is mixing or, at least, hypercyclic.}
\end{remark}
\bigskip


\section{Examples} 
\label{ex}

We start with an obvious example when a graph cannot carry a hypercyclic backward shift
even if we have a freedom in the choice of a measure. 

\begin{example}
{\rm Let $V = \{v_{1,1},v_{1,2}\} \cup \{v_{i,i-1},v_{i,i}, v_{i,i+1}, \ i\ge 2\}$. Assume that
$v_{i,i-1} \to v_{i,i}$, $v_{i-1,i} \to v_{i,i}$, $v_{i,i} \to v_{i+1,i}$, $v_{i,i} \to v_{i,i+1}$, 
and there are no other edges in $G$. Then it is clear that $(B^n f)(v_{i,i-1}) = (B^n f)(v_{i-1,i}) =  (B^{n-1} f)(v_{i,i}) $,
and so $B$ is not hypercyclic.

More generally, if $G$ is a graph with a vertex $v$ satisfying that $|Par(v)| > 1$
and $|Chi(Par(v))| = 1$ (i.e., $Chi(Par(v)) = \{v\}$), then the same argument applies and $B$ is not hypercyclic.
Here we denote by $|E|$ the number of elements in the set $E$.}
\end{example}

\begin{example}
{\rm Let $V = \{v_j\}_{j\ge 1} \cong \N$ and 
$E = \{(v_i, v_{i+1}), \ i\ge 1\} \cup\{(v_2, v_1)\}$. Thus, $V$ is the tree $\N$ with one added edge making a cycle.
We consider the ``Rolewicz operator'' $\alpha B$ (considered for the first time in  \cite{rol}), where $\alpha\in\co$, $|\alpha|>1$, 
on the unweighted space $\ell^p$. By a trivial computation, for a sequence $f=(f_j)_{j\ge 1}$,
we have
$$
B^{2n-1} f = \bigg( \sum_{j=1}^n f_{2j}, \sum_{j=1}^{n+1} f_{2j-1}, f_{2n+2}, f_{2n+3}, \ldots \bigg),
$$
$$
B^{2n} f = \bigg( \sum_{j=1}^{n+1} f_{2j-1}, \sum_{j=1}^{n+1} f_{2j}, f_{2n+3}, f_{2n+4}, \ldots \bigg).
$$
We show that $\alpha B$ is hypercyclic (and even mixing) in $\ell^p$ for $1 < p<\infty$. 
Denote by $Y$ the set of all finite sequences $(x_k)$ such that
$\sum\limits_{k \text{ -- odd}} x_{k} = \sum\limits_{k \text{ -- even}} x_{k} =0$. Clearly, $Y$ is dense in $\ell^p$ 
for  $1 < p<\infty$.  Now let $U$  be a non-empty subset of $\ell^p$ and let $W$ be a neighborhood of zero. 
If $x\in U\cap Y$, then $(aB)^n x = 0$ eventually, which means that the return set $N(U,W)
=\{n: (aB)^n(U) \cap W\ne\emptyset\}$ is cofinite. If we set
$$
y_n = \frac{1}{a^n} \sum_{k=1}^\infty x_k e_{k+n},
$$
then $y_n\to 0$ and $(aB)^n y_n = x$ for each $n\in \N$ which shows that also $N(W,U)$ is
cofinite. We conclude that $aB$ is mixing. }
\end{example}

\begin{remark}
{\rm
Note that any bounded operator $T$ acting on a Banach space $X$ with a Schauder base can be interpreted 
as a {\it weighted} backward shift on some Banach space of sequences indexed by some 
graph. Indeed, if $\{e_n\}_{n=1}^\infty$ is a Schauder basis and
$$
Te_n = \sum_{i=1}^\infty a_i^n e_i,
$$
then we define the graph
$G$ to have vertices $e_n$ and we assume that $e_i\to e_n$ if and only if $a_i^n \ne 0$.
We define the weight on the edge $(e_i, e_n)$ as $a_i^n$. Thus, for $x=\sum_n c_n e_n$ we have 
$Tx = \sum_i \Big(\sum_n a_i^n c_n\Big) e_i$, and so $T$ is a weighted backward shift on the coefficient space of the
basis $\{e_n\}$ equipped with its standard norm (see, e.g., \cite[Chapter 1]{young}). }
\end{remark}

\begin{example}
{\rm In the next example we again consider an unweighted backward shift. 
Let $V = \{v_j\}_{j\ge 1} \cong \N$ and assume that $v_j \to v_{j+1}$ and
$v_j \to v_{j+2}$, $j\ge 1$, and all edges are of this form.
Then we have $(Bf)(v_j) =  f(v_{j+1})+ f(v_{j+2})$ 
and so $B$ essentially coincides with $B_0(I+B_0)$ where $B_0$ is the usual backward shift on $\N$.

Assume that $\mu=\{\mu_n\}$ is a weight such that the backward shift $B_0$ is bounded on $X$
where $X = \ell^p(\N, \mu)$, $1\le p <\infty$, or $X = c_0(\N, \mu)$.
Let $q=\limsup_{n\to \infty} |\mu_n|^{1/n}$. Then for any $s \in \mathbb{C}$ with $|s| <q^{-1}$ we have
$f_s= \{s^n\}_{n\ge 1} \in X $ and $B_0(I+B_0)f_s = s(1+s) f_s$. 
If $q<\frac{1+\sqrt{5}}{2}$, then $q^{-1}(1+q^{-1})>1$, and we have an open set of eigenvalues
$\lambda$ of $B_0(I+B_0)$  with $|\lambda| <1$ and an open set of eigenvalues
$\lambda$ of $B_0(I+B_0)$  with $|\lambda| >1$. Corresponding families of eigenvectors are complete 
and so $B_0(I+B_0)$ is mixing by the Godefroy--Shapiro criterion.

As in Theorem \ref{nxn}, one easily shows that $q_0=\frac{1+\sqrt{5}}{2}$ is the critical value. 
Assume that there is $C>0$ such that $|\mu_n|\ge Cq_0^n.$ We show that in this case $B_0(I+B_0)$ 
is not hypercyclic. Indeed, if  $B_0(I+B_0)$ is hypercyclic, then
for any $\vep\in(0,1/2)$ and $N\in \N$ there exists $f = (f_k)_{k\in \N} \in X$ and $n>N$ such that $\|f\|_X < \vep$
and $\| (B_0(I+B_0))^n f - e_1 \|_X<\vep |\mu_1|$ where $e_1=(1, 0, 0, \dots)$.
On one hand, we have $|f_k| \le C\vep q_0^{-k}$. On the other hand, 
$$
\Big|\sum_{k=0}^{n} C_{n}^k f_{k+n} -1\Big| <\vep,
$$
a contradiction, when $\vep$ is sufficiently small.  This is in contrast with the case
of the operator $I+B_0$ which is weakly mixing on
$\ell^p(\N, \mu)$, $1\le p <\infty$, or $c_0(\N, \mu)$ whenever $B_0$ is bounded \cite[Theorem 8.2]{gp}.  }
\end{example}

\end{document}